\documentclass[11pt]{amsart}

\usepackage{amssymb, amsmath, amsthm}

 \newtheorem{thm}{Theorem}[section]
 \newtheorem{cor}[thm]{Corollary}
 \newtheorem{lem}[thm]{Lemma}
 \newtheorem{prop}[thm]{Proposition}

    \theoremstyle{definition}

  \newtheorem{example}[thm]{Example}
 \theoremstyle{remark}
 \newtheorem{rem}[thm]{Remark}
\numberwithin{equation}{section}

\begin{document}

\title[A construction  of an iterated Ore extension]{A construction  of  an iterated Ore extension}

\author[No-Ho Myung]{No-Ho Myung}
\address{Department of Mathematics\\
         Chungnam  National University\\
          99 Daehak-ro,   Yuseong-gu, Daejeon 34134, Korea}
\email{nhmyung@cnu.ac.kr}             

\author[Sei-Qwon Oh]{Sei-Qwon Oh}
\address{                            
         Department of Mathematics\\
         Chungnam National University  \\
         99 Daehak-ro,   Yuseong-gu, Daejeon 34134, Korea}
\email{sqoh@cnu.ac.kr}             






\begin{abstract}
Let $B$ be a Poisson algebra $\Bbb C[x_1,\ldots, x_k]$ with Poisson bracket such that $$\{x_j,x_i\}=c_{ji}x_ix_j+p_{ji}$$ for all $j>i$, where $c_{ji}\in\Bbb C$ and $p_{ji}\in\Bbb C[x_1,\ldots,x_i]$. Here we obtain an iterated skew polynomial algebra such that its semiclassical limit is equal to $B$  and the results are illustrated by  examples.
\end{abstract}

\maketitle


\section{Introduction}
Recall the star product in \cite[1.1]{Kon}.
Let $R=(R,\{-,-\})$ be a Poisson algebra  and let $Q$ be a quantization of $R$ with a star product $*$.
Then $Q$ is a ${\Bbb C}[[\hbar]]$-algebra $R[[\hbar]]$ such that for any $a,b\in R\subset Q=R[[\hbar]]$,
$$a*b=ab+B_1(a,b)\hbar+B_2(a,b)\hbar^2+\ldots$$ subject to
\begin{equation}\label{PBRACKET}
\{a,b\}=\hbar^{-1}(a*b-b*a)|_{\hbar=0},
\end{equation}
where $B_i:R\times R\longrightarrow R$ are bilinear products. In general, the star product is as follows:
for any $f=\sum_{n\geq0}f_n\hbar^n, g=\sum_{n\geq0}g_n\hbar^n\in Q$
$$(\sum_{n\geq0}f_n\hbar^n)*(\sum_{n\geq0}g_n\hbar^n)=\sum_{k,l\geq0}f_kg_l\hbar^{k+l}+\sum_{k,l\geq0,m\geq1}B_m(f_k,g_l)\hbar^{k+l+m}.
 $$
 It is well-known that we can recover the Poisson algebra $R=Q/\hbar Q$ with Poisson bracket (\ref{PBRACKET}) from $Q$ since $\hbar$ is  a nonzero, nonunit, non-zero-divisor and central element such that $Q/\hbar Q$ is commutative.
  But it seems that the star product in $Q$ is complicate and $Q$ is difficult to  understand at an algebraic point of view since it is too big.
 For instance, if $\lambda$ is a nonzero element of $\Bbb C$ then $\hbar-\lambda$ is a unit in $Q$ and thus $Q/(\hbar-\lambda)Q$ is trivial. Hence it seems that  we need an appropriate $\Bbb F$-subalgebra $A$ of $Q$ such that $A$ contains all generators of $Q$,  $\hbar\in A$ and $A$ is understandable   at an algebraic point of view, where $\Bbb F$ is a subring of $\Bbb C[[\hbar]]$.

Suppose that $A$ is an algebra and let $\hbar\in A$ be a nonzero, nonunit, non-zero-divisor and central element such that
$A/\hbar A$ is commutative. Then $A/\hbar A$ is a nontrivial commutative algebra as well as a Poisson algebra with the Poisson bracket
\begin{equation}\label{PBRACKET2}
\{\overline{a}, \overline{b}\}=\overline{\hbar^{-1}(ab-ba)}
\end{equation}
for $\overline{a}, \overline{b}\in A/\hbar A$. Note that (\ref{PBRACKET}) is equal to (\ref{PBRACKET2}). Further, if  there is an element  $0\neq\lambda\in\Bbb C$ such that $\hbar-\lambda$ is a  nonunit in $A$ then we obtain a  nontrivial algebra $A/(\hbar-\lambda)A$ with the multiplication induced by that of $A$.
The  Poisson algebra $A/\hbar A$ is called a {\it semiclassical limit} of $A$ and the  nontrivial algebra $A/(\hbar-\lambda)A$ is called a {\it deformation} of $A$ or $A/\hbar A$  in \cite[2.1]{Good4}.
The element $\hbar\in A$ inducing the Poisson algebra $A/\hbar A$ is called a {\it regular element} of $A$. Namely, by a regular element  $\hbar\in A$ we mean a nonzero, nonunit, non-zero-divisor and central element of $A$ such that $A/\hbar A$ is commutative.
(An anonymous referee suggested to use the terminology \lq regular element\rq \ while several papers for semiclassical limits were written even though there are many concepts for \lq regular element\rq \ as in \cite{McRo} and \cite{GoWa2}. We hope that a nice terminology for this concept is given.)
In  general, let $A$ be an $\Bbb F$-algebra  generated by $x_1,\ldots, x_n$ with relations $f_1,\ldots, f_r$ and let $\lambda\in \Bbb C$, where $\Bbb F$ is a subring of $\Bbb C[[\hbar]]$ containing $\Bbb C[\hbar]$ and $f_i$ are elements of the free $\Bbb F$-algebra on the set $\{x_1,\ldots, x_n\}$. Assume that $g|_{\hbar=\lambda}$, $f_i|_{\hbar=\lambda}$ make sense mathematically for all $g\in A$ and $i=1,\ldots,r$.
Denote by $A_\lambda$ the $\Bbb C$-algebra generated by $x_1,\ldots, x_n$ with the relations $f_1|_{\hbar=\lambda},\ldots, f_r|_{\hbar=\lambda}$ and let $\varphi$ be the evaluation map from $A$ onto $A_\lambda$ defined by $g\mapsto g|_{\hbar=\lambda}$. Then $\varphi$ is a $\Bbb C$-algebra epimorphism and $A/\ker\varphi\cong A_\lambda$. In particular, if
$\ker \varphi\neq A$ then $A_\lambda$ is nontrivial  and the multiplication of $A_\lambda$ is induced by that of $A$. We still call the nontrivial algebra $A_\lambda$ a {\it deformation} of $A/\hbar A$.

Let $B_k$ be a Poisson $\Bbb C$-algebra $\Bbb C[x_1,\ldots, x_k]$ with Poisson bracket such that for all $j>i$,
\begin{equation}\label{PBRACKET3}
\{x_j,x_i\}=c_{ji}x_ix_j+p_{ji},
\end{equation}
  where $c_{ji}\in\Bbb C$ and $p_{ji}\in\Bbb C[x_1,\ldots,x_i]$.
A main aim of the article is to give how to construct an $\Bbb F$-algebra which is presented by an iterated skew polynomial algebra such that its semiclassical limit is equal to the given Poisson algebra $B_k$.

Let $t$ be an indeterminate and let $\Bbb C[[t-1]]$ be the ring of formal power series over $\Bbb C$ at $t-1$. Namely,
$$\Bbb C[[t-1]]=\left\{\sum_{i=0}^\infty b_i(t-1)^i\ |\ b_i\in\Bbb C\right\}.$$
Note that $\Bbb C[[t-1]]$ is an integral domain, that $\Bbb C[t]\subseteq\Bbb C[[t-1]]$ and that a nonzero element $\sum_{i=0}^\infty b_i(t-1)^i$ is a unit in $\Bbb C[[t-1]]$ if and only if $b_0\neq0$. We assume throughout the article that $\Bbb F$ is a subring  of $\Bbb C[[t-1]]$ containing $\Bbb C[t]$, namely
$$\Bbb C[t]\subseteq \Bbb F\subseteq\Bbb C[[t-1]].$$
Let
 $$A_{k-1}=\Bbb F[x_1][x_2;\beta_2,\nu_2]\ldots[x_{k-1};\beta_{k-1},\nu_{k-1}]$$
 be an iterated skew polynomial $\Bbb F$-algebra and let $\beta_k,\nu_k$ be $\Bbb F$-linear maps from  $A_{k-1}$ into itself.
 In this article, we find  necessary and sufficient conditions for $\beta_k$ and $\nu_k$ such that there exists a skew polynomial algebra $A_k=A_{k-1}[x_k;\beta_k,\nu_k]$ under suitable conditions. (See Lemma~\ref{NECE} and  Theorem~\ref{MAIN}.) Hence, using induction on $k$  repeatedly, we can get iterated skew polynomial algebras from the result. Next we observe that $t-1$ is a regular element of $A_k$ and   find a condition  such that the Poisson algebra $B_k=\Bbb C[x_1,\ldots, x_k]$ with Poisson bracket (\ref{PBRACKET3}) is Poisson isomorphic to the  semiclassical limit $A_k/(t-1) A_k$. (See Corollary~\ref{AD}  and \cite[\S 2]{ChOh3}.) Finally we give examples illustrating the results.

Recall several basic terminologies.
(1) Given an $\Bbb F$-endomorphism $\beta$ on an $\Bbb F$-algebra $R$, an $\Bbb F$-linear map $\nu$ is said to be a {\it left $\beta$-derivation} on $R$ if
$\nu(ab)=\beta(a)\nu(b)+\nu(a)b$ for all $a,b\in R$. For such a pair $(\beta,\nu)$, we denote by $R[z;\beta,\nu]$ the skew polynomial $\Bbb F$-algebra. Refer to \cite[Chapter 2]{GoWa2} for details of a skew polynomial algebra.

(2) A commutative $\Bbb C$-algebra $R$ is said to be a {\it Poisson algebra} if there exists a bilinear product $\{-,-\}$ on $R$, called a {\it Poisson bracket}, such that $(R, \{-,-\})$ is a Lie algebra with $\{ab,c\}=a\{b,c\}+\{a,c\}b$ for all $a,b,c\in R$.
A derivation $\alpha$ on $R$ is said to be a {\it Poisson derivation} if $\alpha(\{a,b\})=\{\alpha(a),b\}+\{a,\alpha(b)\}$ for all $a,b\in R$. Let $\alpha$ be a Poisson derivation on $R$ and let $\delta$ be a derivation on $R$ such that
\begin{equation}\label{SKEW}
\delta(\{a,b\})-\{\delta(a),b\}-\{a,\delta(b)\}=\alpha(a)\delta(b)-\delta(a)\alpha(b)
\end{equation}
for all $a,b\in R$. By \cite[1.1]{Oh8}, the commutative polynomial $\Bbb C$-algebra $R[z]$ is a Poisson  algebra  with Poisson bracket $\{z,a\}=\alpha(a)z+\delta(a)$ for all $a\in R$. Such a Poisson polynomial algebra
$R[z]$ is denoted by $R[z;\alpha,\delta]_p$ in order to distinguish it from skew polynomial algebras. If $\alpha=0$ then
we write $R[z;\delta]_p$ for $R[z;0,\delta]_p$ and if $\delta=0$ then we write $R[z;\alpha]_p$ for $R[z;\alpha,0]_p$.


\section{A construction of an iterated skew polynomial algebra}
Set $A_1=\Bbb F[x_1]$ and let $A_{n}$, $n>1$, be an iterated skew polynomial $\Bbb F$-algebra
$$A_{n}=\Bbb F[x_1][x_2;\beta_2,\nu_2]\ldots[x_{n};\beta_{n},\nu_{n}].$$
By monomials in $A_n$ we mean  finite products of $x_i$'s together with the unity 1. A monomial $X$ is said to be {\it standard} if $X$ is of the form
$$X=1\text{ or }X=x_{i_1}x_{i_2}\cdots x_{i_k}\qquad (1\leq i_1\leq i_2\leq\ldots\leq i_k\leq n).$$
Note that the set of all standard monomials of $A_n$ forms an $\Bbb F$-basis.

Let $\beta$ and $\nu$ be $\Bbb F$-linear maps from an $\Bbb F$-algebra $R$ into itself.
The following lemma is well known, e.g. see \cite[p.177]{Jac}.
\begin{lem}\label{hhh}
The following conditions are equivalent:\\
(1) The $\Bbb F$-linear map $\phi:R\rightarrow M_2(R)$ by
$$\phi(r)=\left(\begin{matrix}\beta(r)&\nu(r)\\ 0&r \end{matrix}\right)$$
for all $r\in R$, is an $\Bbb F$-algebra homomorphism\\
(2) $\beta$ and $\nu$ are an endomorphism and a left $\beta$-derivation on $R$ respectively.
\end{lem}

In an iterated skew polynomial $\Bbb F$-algebra
$$A_{k-1}=\Bbb F[x_1][x_2;\beta_2,\nu_2]\ldots[x_{k-1};\beta_{k-1},\nu_{k-1}],$$
assume that $\beta_j,\nu_j$ $(j=2,\ldots, k-1)$ satisfy
\begin{eqnarray}
&\beta_j(x_i)=a_{ji}x_i, \ a_{ji}\in\Bbb F\ \ &( 1\leq i<j< k)\label{hom}\\
&\nu_j(x_i)=u_{ji}\in A_{i}\ \ \ \ \ \ \ \ &(1\leq i<j< k).\label{der0}
\end{eqnarray}

We are going to construct a skew polynomial $\Bbb F$-algebra
$$A_k=A_{k-1}[x_k;\beta_k,\nu_k]=\Bbb F[x_1][x_2;\beta_2,\nu_2]\ldots[x_{k};\beta_{k},\nu_{k}]$$ such that $\beta_k$, $\nu_k$ satisfy the following  conditions
\begin{eqnarray}
&\beta_k(1)=1,\ \beta_k(x_i)=a_{ki}x_i,\  a_{ki}\in\Bbb F\ \ &(1\leq i\leq k-1),\label{AU}\\
&\nu_k(1)=0,  \ \nu_k(x_{i})=u_{ki}\in A_i \ \ \ \ \ \ \ \ \  &(1\leq i\leq k-1).\label{?}
\end{eqnarray}

The following statement gives us  necessary conditions for the existence of the skew polynomial $\Bbb F$-algebra $A_k=A_{k-1}[x_k;\beta_k,\nu_k]$ over $A_{k-1}$.

\begin{lem}\label{NECE}
If there exists a skew polynomial $\Bbb F$-algebra
$A_k=A_{k-1}[x_k;\beta_k,\nu_k]$ such that $\beta_k$, $\nu_k$ are subject to (\ref{AU}), (\ref{?}) then
$\beta_k$, $\nu_k$  satisfy the following conditions
\begin{eqnarray}
&\beta_k(u_{ji})=a_{kj}a_{ki}u_{ji}\qquad\qquad\qquad     &(1\leq i<j<k),\label{ENDO1}\\
&a_{kj}x_ju_{ki}+u_{kj}x_i=a_{ji}u_{ki}x_j+a_{ki}a_{ji}x_iu_{kj}+\nu_k(u_{ji}) &(1\leq i<j<k).\label{ENDO2}
\end{eqnarray}
\end{lem}

\begin{proof}
Let $1\leq i<j\leq k-1$.
Since $\beta_k$ is an $\Bbb F$-algebra endomorphism, we have that
$$\beta_k(x_jx_i)= \beta_k(\beta_j(x_i)x_j+\nu_j(x_i))=a_{kj}a_{ki}a_{ji}x_ix_j+\beta_k(u_{ji})$$
and
$$
\begin{aligned}
\beta_k(x_jx_i)&=\beta_k(x_j)\beta_k(x_i)=a_{kj}a_{ki}x_jx_i\\
&=a_{kj}a_{ki}(\beta_j(x_i)x_j+\nu_j(x_i))
=a_{kj}a_{ki}a_{ji}x_ix_j+a_{kj}a_{ki}u_{ji}
\end{aligned}
$$
by (\ref{hom})-(\ref{?}).
Hence we get (\ref{ENDO1}).

Similarly, since $\nu_k$ is a left $\beta_k$-derivation, we have that
$$\nu_k(x_jx_i)= \nu_k(x_j)x_i+\beta_k(x_j)\nu_k(x_i)=u_{kj}x_i+a_{kj}x_ju_{ki}$$
and
$$\begin{aligned}
\nu_k(x_jx_i)&=\nu_k(\beta_j(x_i)x_j+\nu_j(x_i))=\nu_k(a_{ji}x_ix_j+u_{ji})\\
&=a_{ji}(\nu_k(x_i)x_j+\beta_k(x_i)\nu_k(x_j))+\nu_k(u_{ji})\\
&=a_{ji}u_{ki}x_j + a_{ki}a_{ji}x_iu_{kj} +\nu_k(u_{ji})
\end{aligned}$$
by (\ref{hom})-(\ref{?}).
Hence we get (\ref{ENDO2}).
\end{proof}

\begin{lem}
For $1\leq i<j\leq k-1$, let all $\beta_j,\nu_j, a_{ji}, u_{ji}$ satisfy (\ref{hom}), (\ref{der0}). Let
$\beta_k, \nu_k$ be $\Bbb F$-linear maps from $A_{k-1}$ into itself subject to the conditions (\ref{AU}) and  (\ref{?}).
If $\beta_k$ and $\nu_k$ satisfy (\ref{ENDO1}) and (\ref{ENDO2}) then  the following conditions hold.
\begin{equation}\label{endo}
\beta_k(x_j)\beta_k(x_i)=\beta_k\beta_j(x_i)\beta_k(x_j)+\beta_k\nu_j(x_i),
\end{equation}
\begin{equation}\label{der1}
\beta_k(x_j)\nu_k(x_i)+\nu_k(x_j)x_i=\beta_k\beta_j(x_i)\nu_k(x_j)+\nu_k\beta_j(x_i)x_j+\nu_k\nu_j(x_i)
\end{equation}
\end{lem}

\begin{proof}
Since $A_{k-1}$ is an iterated skew polynomial $\Bbb F$-algebra, the equations (\ref{endo}) and (\ref{der1}) follow from  (\ref{ENDO1}) and (\ref{ENDO2}), respectively, by (\ref{hom})-(\ref{?}).
\end{proof}

In the following theorem, we see that  (\ref{ENDO1}) and (\ref{ENDO2}) are  sufficient conditions for the existence of the skew polynomial $\Bbb F$-algebra $A_k=A_{k-1}[x_k;\beta_k,\nu_k]$ over $A_{k-1}$.

\begin{thm}\label{MAIN}
For $1\leq i<j\leq k-1$, let all $\beta_j,\nu_j, a_{ji}, u_{ji}$ satisfy (\ref{hom}), (\ref{der0}).
Given  $\Bbb F$-linear maps  $\beta_k, \nu_k$    from $A_{k-1}$ into itself subject to
(\ref{AU}), (\ref{?}), if $\beta_k$ and $\nu_k$ satisfy the  conditions (\ref{ENDO1}), (\ref{ENDO2}) then  there exists an iterated skew polynomial $\Bbb F$-algebra
$$A_k=A_{k-1}[x_k;\beta_k,\nu_k]=\Bbb F[x_1][x_2;\beta_2,\nu_2]\ldots[x_{k};\beta_{k},\nu_{k}].$$
\end{thm}

\begin{proof}
It is enough to show that there exist an $\Bbb F$-algebra endomorphism $\beta_k$ on $A_{k-1}$ and a left $\beta_k$-derivation $\nu_k$ subject to  the conditions (\ref{AU}) and (\ref{?}).
Note that the set of all standard monomials forms an $\Bbb F$-basis of $A_{k-1}$. For any standard monomials $x_{i_1}\cdots x_{i_r}\in A_{k-1}$,  define  $\Bbb F$-linear maps $\beta_k$ and $\nu_k$ from $A_{k-1}$ into itself  by
\begin{eqnarray}
\beta_k(1)=1, \ &\beta_k(x_{i_1}\cdots x_{i_r})&=(a_{ki_1}x_{i_1})\cdots(a_{ki_r}x_{i_r}),\label{beta}\\
\nu_k(1)=0,\ &\nu_k(x_{i_1}\cdots x_{i_r})&=\sum_{\ell=1}^{r}(a_{ki_1}x_{i_1})\cdots (a_{ki_{\ell-1}}x_{i_{\ell-1}})u_{ki_\ell}(x_{i_{\ell+1}}\cdots x_{i_r}).  \label{nu}
\end{eqnarray}
Observe that  these $\Bbb F$-linear maps $\beta_k$ and $\nu_k$  satisfy (\ref{AU}) and (\ref{?}).
We will  show that the map $\beta_k$ defined by (\ref{beta}) is an $\Bbb F$-algebra endomorphism and the map $\nu_k$ defined by (\ref{nu}) is a left $\beta_k$-derivation by using Lemma~\ref{hhh}.

Let $\Bbb F\langle S_{k-1}\rangle$ be the free $\Bbb F$-algebra on the set $S_{k-1}=\{x_1,\ldots,x_{k-1}\}$. Define an $\Bbb F$-algebra homomorphism $f:\Bbb F\langle S_{k-1}\rangle\rightarrow M_2(A_{k-1})$ by
$$f(x_i)=\left(\begin{matrix}\beta_k(x_i)&\nu_k(x_i)\\ 0&x_i \end{matrix}\right)\qquad(1\leq i<k).$$
Let us show that
\begin{equation}\label{FU}
f(\nu_j(x_i))=\left(\begin{matrix}\beta_k\nu_j(x_i)&\nu_k\nu_j(x_i)\\ 0&\nu_j(x_i) \end{matrix}\right)
\end{equation}
for $1\leq i<j<k$.
For any standard monomial $X=x_{i_1}\cdots x_{i_r}$ in $A_{k-1}$, by (\ref{beta}) and (\ref{nu}),
$$\begin{aligned}
\nu_k(X)&=\sum_{\ell=1}^r \beta_k(x_{i_1}\cdots x_{i_{\ell-1}})\nu_k(x_{i_\ell})(x_{i_{\ell+1}}\cdots x_{i_r})\\
&=\sum_{\ell=1}^{r-1} \beta_k(x_{i_1}\cdots x_{i_{\ell-1}})\nu_k(x_{i_\ell})(x_{i_{\ell+1}}\cdots x_{i_r}) + \beta_k(x_{i_1}\cdots x_{i_{r-1}})\nu_k(x_{i_r})\\
&=\nu_k(x_{i_1}\cdots x_{i_{r-1}})x_{i_r}+\beta_k(x_{i_1}\cdots x_{i_{r-1}})\nu_k(x_{i_r}).
\end{aligned}$$
In particular, if $Xx_j$ is standard (thus $i_r\leq j$) then
\begin{equation}\label{NU}
\nu_k(Xx_j)=\beta_k(X)\nu_k(x_j)+\nu_k(X)x_j.
\end{equation}
Let us verify first that
\begin{equation}\label{FU1}
f(X)=\left(\begin{matrix}\beta_k(X)&\nu_k(X)\\ 0&X \end{matrix}\right)
\end{equation}
for any  standard monomial $X=x_{i_1}\cdots x_{i_r}$ in $A_{k-1}$ of length $r$. We proceed by induction on $r$. If $r=1$ then (\ref{FU1}) is true trivially. Assume that $r>1$ and that (\ref{FU1}) holds for any standard monomial of length $<r$.
Set $Y=x_{i_1}\cdots x_{i_{r-1}}$. Then $Y$ is a standard monomial of length $r-1$ and $X=Yx_{i_r}$. Thus (\ref{FU1}) holds as follows:
$$\begin{aligned}
f(X)&=f(Yx_{i_r})=f(Y)f(x_{i_r})&&\\
&=\left(\begin{matrix}\beta_k(Y)&\nu_k(Y)\\ 0&Y \end{matrix}\right)\left(\begin{matrix}\beta_k(x_{i_r})&\nu_k(x_{i_r})\\ 0&x_{i_r} \end{matrix}\right)&&(\text{by induction hypothesis})\\
&=\left(\begin{matrix}\beta_k(Y)\beta_k(x_{i_r})&\beta_k(Y)\nu_k(x_{i_r})+\nu_k(Y)x_{i_r}\\ 0&Yx_{i_r} \end{matrix}\right)&&\\
&=\left(\begin{matrix}\beta_k(X)&\nu_k(X)\\ 0&X \end{matrix}\right).&&(\text{by }(\ref{beta}),(\ref{NU}))
\end{aligned}$$

Let $\nu_j(x_i)=\sum_{\ell} b_{\ell} X_\ell$, where all $b_\ell\in\Bbb F$ and $X_\ell$ are standard monomials of $A_{i}$.
Since $f$ is an $\Bbb F$-algebra homomorphism, we have
$$\begin{aligned}
f(\nu_j(x_i))&=\sum_\ell b_\ell f(X_\ell)&&\\
&=\sum_\ell b_\ell\left(\begin{matrix}\beta_k(X_\ell)&\nu_k(X_\ell) \\
0&X_\ell\end{matrix}\right)&&(\text{by }(\ref{FU1}))\\
&=\left(\begin{matrix}\beta_k(\sum_\ell b_\ell X_\ell)&\nu_k(\sum_\ell b_\ell X_\ell) \\
0&\sum_\ell b_\ell X_\ell\end{matrix}\right)&&\\
&=\left(\begin{matrix}\beta_k\nu_j(x_i)&\nu_k\nu_j(x_i)\\ 0&\nu_j(x_i) \end{matrix}\right).&&
\end{aligned}$$
Thus (\ref{FU}) holds.

Note that $A_{k-1}$ is an $\Bbb F$-algebra generated by $x_1,\ldots, x_{k-1}$ with relations
$$x_jx_i-\beta_j(x_i)x_j-\nu_j(x_i)\qquad (1\leq i<j<k).$$
Namely, $A_{k-1}$ is isomorphic to the $\Bbb F$-algebra $\Bbb F\langle S_{k-1}\rangle/I$, where $I$ is the ideal generated by    $$x_jx_i-\beta_j(x_i)x_j-\nu_j(x_i)\ \ \ \ (1\leq i<j<k).$$
Since $f$ is an $\Bbb F$-algebra homomorphism, it is easy to check that $I\subseteq\text{ker}f$ by (\ref{endo}), (\ref{der1}) and (\ref{FU}).
Hence there exists an $\Bbb F$-algebra homomorphism $\phi:A_{k-1}\rightarrow M_2(A_{k-1})$ such that
$$\phi(x_i)=\left(\begin{matrix}\beta_k(x_i)&\nu_k(x_i)\\ 0&x_i \end{matrix}\right)$$
for $1\leq i<k$.
By Lemma~\ref{hhh}, $\beta_k$ is an $\Bbb F$-algebra endomorphism on $A_{k-1}$ and $\nu_k$ is a left $\beta_k$-derivation on $A_{k-1}$ as claimed.
\end{proof}

\begin{rem}\label{REM}
Retain the notations of Theorem~\ref{MAIN}.

(1) If $a_{ki}\neq0$ for all $1\leq i<k$ then $\beta_k$ is a monomorphism.

(2)  If $u_{ji}=0$ for all $1\leq i<j\leq k$ then (\ref{ENDO1}) and (\ref{ENDO2}) hold trivially.

(3) If $A_{k-1}$ is commutative and $a_{ki}=1$ for all $1\leq i\leq k-1$ then (\ref{ENDO1}) and (\ref{ENDO2}) hold.
\end{rem}

\begin{proof}
(1) Note that  $\beta_i$, $\nu_i$ are $\Bbb F$-linear for all $i=1,\ldots,k$. Let $f=\sum_i a_iX_i\in A_{k-1}$, where $a_i\in\Bbb F$ and $X_i$ are standard monomials for all $i$, and suppose that $\beta_k(f)=0$. Then
$\beta_k(X_i)=b_iX_i$ for some $0\neq b_i\in\Bbb F$ by (\ref{beta}) and thus $$0=\beta_k(f)=\sum_i a_ib_iX_i.$$ It follows that all
$a_i=0$ since the standard monomials of $A_k$ form an $\Bbb F$-basis. Thus $f=0$.

(2) Trivial.

(3) Since $A_{k-1}$ is commutative, $u_{ji}=0$  and $a_{ji}=1$ for all $1\leq i<j\leq k-1$ and thus (\ref{ENDO1}) and (\ref{ENDO2}) hold.
\end{proof}

\begin{thm}\label{LIMIT}
Let $A_k=\Bbb F[x_1][x_2;\beta_2,\nu_2]\ldots[x_k;\beta_k,\nu_k]$ be the iterated skew polynomial $\Bbb F$-algebra in Theorem~\ref{MAIN}. Suppose that $\Bbb F/(t-1)\Bbb F$ is isomorphic to $\Bbb C$, that $t-1$ is a  nonunit and non-zero-divisor in $A_k$ and that
\begin{equation}\label{CON1}
a_{ji}-1\in (t-1)\Bbb F, \ \ \nu_j(x_i)\in (t-1)A_k
\end{equation}
for all  $1\leq i<j\leq k$. Then $t-1$ is a regular element of $A_k$ and  the semiclassical limit $\overline{A}_k=A_k/( t-1)A_k$ is Poisson isomorphic to an iterated Poisson polynomial $\Bbb C$-algebra
$$\Bbb C[x_1][x_2;\alpha_2,\delta_2]_p\ldots[x_k;\alpha_k,\delta_k]_p,$$
where
\begin{equation}\label{REL1}
\alpha_j(x_i)=\left(\frac{da_{ji}}{dt}|_{t=1}\right)x_i, \ \ \delta_j(x_i)=\frac{d\nu_j(x_i)}{dt}|_{t=1}
\end{equation}
for all $1\leq i<j\leq k$. (Derivatives are formal derivatives of power series in $t-1$.)
\end{thm}

\begin{proof}
Note that $A_k$ is generated by $x_1,\ldots, x_k$ and that $t-1\in\Bbb F\subset A_k$. Hence $t-1$ is a nonzero central  element of $A_k$.
Since
\begin{equation}\label{COM}
x_jx_i-x_ix_j=\beta_j(x_i)x_j+\nu_j(x_i)-x_ix_j=(a_{ji}-1)x_ix_j+\nu_j(x_i)\in (t-1)A_k
\end{equation}
by (\ref{CON1}), $\overline{A}_k$ is a commutative $\Bbb C$-algebra and thus $t-1$ is a regular element of $A_k$. Moreover we have
$$\begin{aligned}
\{\overline{x}_j,\overline{x}_i\}&=\overline{(t-1)^{-1}(x_jx_i-x_ix_j)}&&\\
&=\overline{\left(\frac{a_{ji}-1}{t-1}\right)x_ix_j+\left(\frac{\nu_j(x_i)}{t-1}\right)}&&(\text{by } (\ref{COM}))\\
&=\left(\frac{da_{ji}}{dt}|_{t=1}\right)\overline{x}_i\overline{x}_j+\overline{\left(\frac{d\nu_j(x_i)}{dt}|_{t=1}\right)}&&(\text{by } (\ref{CON1}))
\end{aligned}$$
for all $1\leq i<j\leq k$. Hence the result follows.
\end{proof}

For each positive integer $k$, we will write $B_k$ for the commutative polynomial ring $\Bbb C[x_1,\ldots,x_k]$.

\begin{lem}
Let $B_k=\Bbb C[x_1,\ldots,x_k]$ be a Poisson algebra satisfying the following condition: for any $1\leq i<j\leq k$,
\begin{equation}\label{MOC1}
\{x_j,x_i\}=c_{ji}x_ix_j+p_{ji}
\end{equation}
for some $c_{ji}\in\Bbb C, p_{ji}\in B_i.$ Then $B_k$ is an iterated Poisson polynomial algebra of the form
\begin{equation}\label{MOC2}
B_k=\Bbb C[x_1][x_2;\alpha_2,\delta_2]_p\ldots[x_k;\alpha_k,\delta_k]_p,
\end{equation}
where $$\alpha_j(x_i)=c_{ji}x_i,\ \ \ \  \delta_j(x_i)=p_{ji}.$$

Conversely, if $B_k$ is an iterated Poisson polynomial algebra of the form (\ref{MOC2}) then $B_k$ is a Poisson algebra satisfying
the condition (\ref{MOC1}).
\end{lem}

\begin{proof}
Suppose that $B_k$ is a Poisson algebra satisfying
the condition (\ref{MOC1}).
Define derivations $\alpha_k$, $\delta_k$ on $B_{k-1}$ by
$$\alpha_k=\sum_{i=1}^{k-1}c_{ki}\frac{\partial}{\partial x_i},\ \ \delta_k=\sum_{i=1}^{k-1}p_{ki}\frac{\partial}{\partial x_i}.$$
Then $\alpha_k$ is a Poisson derivation, $\delta_k$ is a derivation and the pair $(\alpha_k,\delta_k)$ satisfies (\ref{SKEW}) by \cite[1.1]{Oh8} since $B_k$ is a Poisson algebra.
Thus $B_k$ is a Poisson polynomial algebra
$$B_k=\Bbb C[x_1,\ldots,x_{k-1}][x_k;\alpha_{k},\delta_{k}]_p$$
over the Poisson subalgebra $B_{k-1}=\Bbb C[x_1,\ldots,x_{k-1}]$. The result follows by induction on $k$.

Conversely, if $B_k$ is an iterated Poisson polynomial algebra of the form (\ref{MOC2}) then $B_k$ is clearly a Poisson algebra satisfying
the condition (\ref{MOC1}).
\end{proof}

\begin{cor}\label{AD}
Let $B_k$ be an iterated Poisson polynomial $\Bbb C$-algebra
$$B_k=\Bbb C[x_1][x_2;\alpha_2,\delta_2]_p\ldots[x_k;\alpha_k,\delta_k]_p$$
such that
$$\begin{aligned}
\alpha_j(x_i)=c_{ji}x_i \ (c_{ji}\in\Bbb C),\ \ \ \delta_j(x_i)&\in\Bbb C[x_1, \ldots , x_i]\ \ \
\end{aligned}$$
for all $1\leq i<j\leq k$ and let
$$a_{ji}\in \Bbb F,\ \ u_{ji}\in \Bbb F[x_1, \ldots, x_i]$$
such that
\begin{equation}\label{MDER}
\begin{aligned}
&a_{ji}-1\in(t-1)\Bbb F,&  \frac{da_{ji}}{dt}|_{t=1}&=c_{ji},\\
&u_{ji}\in (t-1)\Bbb F[x_1, \ldots, x_i],&  \frac{d u_{ji}}{d t}|_{t=1}&=[\delta_j(x_i)],
\end{aligned}
\end{equation}
where $[\delta_j(x_i)]$ is the $\Bbb C$-linear combination of $\delta_j(x_i)$ by standard monomials of $x_1, \ldots, x_i$.
Set  $A_1=\Bbb F[x_1]$. Suppose that $\Bbb F/(t-1)\Bbb F$ is isomorphic to $\Bbb C$ and that $t-1$ is a nonunit and non-zero-divisor of an iterated skew polynomial $\Bbb F$-algebra
$$A_{k-1}=\Bbb F[x_1][x_2;\beta_2,\nu_2]\ldots[x_{k-1};\beta_{k-1},\nu_{k-1}]$$
 such that all $\beta_j,\nu_j$ satisfy (\ref{hom}) and (\ref{der0}).
 If $\Bbb F$-linear maps $\beta_k,\nu_k$ on $A_{k-1}$ subject to (\ref{AU}) and (\ref{?}) satisfy (\ref{ENDO1}) and (\ref{ENDO2}) then
   there exists an iterated skew polynomial $\Bbb F$-algebra
 $$A_k=A_{k-1}[x_k;\beta_k,\nu_k]=\Bbb F[x_1][x_2;\beta_2,\nu_2]\ldots[x_{k-1};\beta_{k-1},\nu_{k-1}][x_k;\beta_k,\nu_k]$$
  and $t-1$ is a regular element of $A_k$
 such that $B_k$ is Poisson isomorphic to the semiclassical limit $A_k/(t-1)A_k$.
\end{cor}

\begin{proof}
By Theorem~\ref{MAIN}, there exists a skew polynomial $\Bbb F$-algebra $A_k=A_{k-1}[x_k;\beta_k,\nu_k]$. Since $t-1$ is still a nonunit and non-zero-divisor in $A_k$, it is a regular element of $A_k$ and the semiclassical limit $A_k/(t-1)A_k$ is Poisson isomorphic to $B_k$ by Theorem~\ref{LIMIT}.
\end{proof}

\section{Examples}
In this section, we give examples which illustrate that $A_k$ is an iterated skew polynomial $\Bbb F$-algebra such that $A_k/(t-1)A_k$ is Poisson isomorphic to a given Poisson algebra $B_k$.
The first four examples appearing in \cite{LuWaZh} are Poisson Hopf algebras presented by iterated Poisson polynomial algebras.
We are interested only in their Poisson structures because we have not found a formal way to give Hopf structures in their deformations yet.

\begin{example}\label{PHA1}
In \cite[Example 3.2]{LuWaZh}, $B=\Bbb C[x_1,x_2,x_3]$ is a Poisson algebra with the Poisson bracket
$$\{x_2,x_1\}=0,\,\,\{x_3,x_1\}=\lambda_{11}x_1,\,\,\{x_3,x_2\}=\lambda_{21}x_1+\lambda_{22}x_2,$$
where $\lambda_{\ell m}\in\Bbb C$.
Observe that $B$ is a Poisson polynomial $\Bbb C$-algebra
$$B=\Bbb C[x_1,x_2][x_3;\delta_3]_p,$$
where
$$\delta_3(x_1)=\lambda_{11}x_1,\,\,\delta_3(x_2)=\lambda_{21}x_1+\lambda_{22}x_2.$$

Set $\Bbb F=\Bbb C[t]$ and
\begin{equation}\label{S004}
a_{31}=a_{32}=1,\,\,u_{31}=f_{11}\lambda_{11}x_1\in\Bbb F[x_1],\,\,u_{32}=f_{21}\lambda_{21}x_1+f_{22}\lambda_{22}x_2\in\Bbb F[x_1,x_2],
\end{equation}
where $f_{\ell m}\in(t-1)\Bbb F$ with $\frac{df_{\ell m}}{dt}|_{t=1}=1$, for example, $f_{\ell m}=(t-1)t^{N_{\ell m}}$ for some nonnegative integer $N_{\ell m}$.
By Remark~\ref{REM}(3), the $\Bbb F$-linear maps $\beta_3$ and $\nu_3$ on $\Bbb F[x_1,x_2]$ defined by
$$\beta_3(x_i)=x_i,\,\,\nu_3(x_i)=u_{3i}\,\,(i=1,2)$$
satisfy (\ref{ENDO1}) and (\ref{ENDO2}). Hence, by Theorem~\ref{MAIN}, there exists a skew polynomial $\Bbb F$-algebra
$$A=\Bbb F[x_1,x_2][x_3;\nu_3].$$
Moreover $t-1$ is a regular element of $A$ and thus $B$ is Poisson isomorphic to the semiclassical limit $A/(t-1)A$ of $A$ by Corollary~\ref{AD}
since all $a_{ji}$, $u_{ji}$ satisfy (\ref{MDER}).
\end{example}

\begin{example}
In \cite[Example 3.3]{LuWaZh}, $B=\Bbb C[x_1,x_2,x_3,x_4]$ is a Poisson algebra with the Poisson bracket
$$\begin{aligned}
\{x_2,x_1\}&=\{x_3,x_1\}=\{x_3,x_2\}=0,\\
\{x_4,x_1\}&=\lambda_{11}x_1,\\
\{x_4,x_2\}&=\lambda_{21}x_1+\lambda_{22}x_2,\\
\{x_4,x_3\}&=\lambda_{31}x_1+\lambda_{32}x_2+(\lambda_{11}+\lambda_{22})x_3,
\end{aligned}$$
where $\lambda_{\ell m}\in\Bbb C$.
Observe that $B$ is a Poisson polynomial $\Bbb C$-algebra
$$B=\Bbb C[x_1,x_2,x_3][x_4;\delta_4]_p,$$
where
$$\delta_4(x_1)=\lambda_{11}x_1,\,\,\delta_4(x_2)=\lambda_{21}x_1+\lambda_{22}x_2,\,\,\delta_4(x_3)=\lambda_{31}x_1+\lambda_{32}x_2+(\lambda_{11}+\lambda_{22})x_3.$$

Set $\Bbb F=\Bbb C[t]$ and
\begin{equation}\label{S005}
\begin{aligned}
&a_{41}=a_{42}=a_{43}=1,\\
&u_{41}=f_{11}\lambda_{11}x_1\in\Bbb F[x_1],\\
&u_{42}=f_{21}\lambda_{21}x_1+f_{22}\lambda_{22}x_2\in\Bbb F[x_1,x_2],\\
&u_{43}=f_{31}\lambda_{31}x_1+f_{32}\lambda_{32}x_2+(f_{11}\lambda_{11}+f_{22}\lambda_{22})x_3\in\Bbb F[x_1,x_2,x_3],
\end{aligned}
\end{equation}
where $f_{\ell m}\in(t-1)\Bbb F$ with $\frac{df_{\ell m}}{dt}|_{t=1}=1$.
By Remark~\ref{REM}(3), the $\Bbb F$-linear maps $\beta_4$ and $\nu_4$ on $\Bbb F[x_1,x_2,x_3]$ subject to
$$\beta_4(x_i)=x_i,\,\,\nu_4(x_i)=u_{4i}\,\,(i=1,2,3)$$
satisfy (\ref{ENDO1}) and (\ref{ENDO2}).
Hence, by Theorem~\ref{MAIN}, there exists a skew polynomial $\Bbb F$-algebra
$$A=\Bbb F[x_1,x_2,x_3][x_4;\nu_4].$$
Moreover $t-1$ is a regular element of $A$ and thus $B$ is Poisson isomorphic to the semiclassical limit $A/(t-1)A$ of $A$ by Corollary~\ref{AD}
since all $a_{ji}$, $u_{ji}$ satisfy (\ref{MDER}).
\end{example}

\begin{example}
In \cite[Example 3.4]{LuWaZh}, $C=\Bbb C[g^{\pm1},x]$ is a Poisson algebra with the Poisson bracket
$$\{x,g\}=\lambda gx,$$
where $\lambda\in\Bbb Z$.
Let $D=\Bbb C[g,h,x]$. Replacing $g^{-1}$ in $C$ by $h$ in $D$, $D$ is a Poisson algebra with the Poisson bracket
$$\{g,h\}=0,\,\,\{x,g\}=\lambda gx,\,\,\{x,h\}=-\lambda hx,$$
namely $D=\Bbb C[g,h][x;\alpha]_p$ is a Poisson algebra by \cite[1.1]{Oh8},
where $\alpha=\lambda g\frac{\partial}{\partial g}-\lambda h\frac{\partial}{\partial h}$ in $\Bbb C[g,h]$.
Note that the ideal $(gh-1)D$ is a Poisson ideal such that $D/(gh-1)D$ is Poisson isomorphic to $C$.

Set $\Bbb F=\Bbb C[t, t^{-1}]$ and $a=t^{\lambda}$.
By Remark~\ref{REM}(2) and Theorem~\ref{MAIN},
there exists a skew polynomial $\Bbb F$-algebra $A=\Bbb F[g,h][x;\beta]$ such that $gh-1$ is a central element in $A$, where
$$\beta(g)=ag,\,\,\beta(h)=a^{-1}h.$$
Set $B=A/(gh-1)A$ and note that $t-1$ is a regular element of $A$ and $B$.
The semiclassical limit $A/(t-1)A$ is Poisson isomorphic to $D$ by Corollary~\ref{AD}
since
$$a-1\in(t-1)\Bbb F,\,\,\frac{da}{dt}|_{t=1}=\lambda,\,\,a^{-1}-1\in(t-1)\Bbb F,\,\,\frac{da^{-1}}{dt}|_{t=1}=-\lambda$$
and the semiclassical limit $B/(t-1)B$ is Poisson isomorphic to $C$.
\end{example}

\begin{example}
In \cite[Example 3.7]{LuWaZh}, $C=\Bbb C[E,F,K^{\pm1}]$ is a Poisson algebra with the Poisson bracket
$$\begin{aligned}
\{E,K\}&=-2KE,\\
\{F,K\}&=2KF,\\
\{F,E\}&=\frac{1}{2}(K^{-1}-K).
\end{aligned}$$

Set $D=\Bbb C[E,F,H,K]$. Replacing $K^{-1}$ in $C$ by $H$ in $D$, it is observed that $D$ is a Poisson algebra with Poisson bracket
$$\begin{aligned}
\{H,K\}&=0,   &\{E,H\}&=2HE,\\
\{E,K\}&=-2KE,&\{F,H\}&=-2HF,\\
\{F,K\}&=2KF, &\{F,E\}&=\frac{1}{2}(H-K)
\end{aligned}$$
and that the ideal $(HK-1)D$ is a Poisson ideal such that $D/(HK-1)D$ is Poisson isomorphic to $C$.
In fact, $D$ is an iterated Poisson polynomial $\Bbb C$-algebra
$$D=\Bbb C[H, K][E;\alpha_3]_p[F;\alpha_4,\delta_4]_p,$$
where
$$\begin{aligned}
\alpha_3(H)&=2H, &\alpha_3(K)&=-2K,&&\\
\alpha_4(H)&=-2H,&\alpha_4(K)&=2K,&\alpha_4(E)&=0,\\
\delta_4(H)&=0,  &\delta_4(K)&=0,&\delta_4(E)&=\frac{1}{2}(H-K).
\end{aligned}$$

Set $\Bbb F=\Bbb C[t,t^{-1}]$ and $s=\sum_{i\ge0}(1-t)^i\in\Bbb C[[t-1]]$.
Since $ts=s-(1-t)s=1$ in $\Bbb C[[t-1]]$, we have that $t^{-1}=s$ and thus $\Bbb C[t]\subset\Bbb F\subset\Bbb C[[t-1]]$.
Set
\begin{equation}\label{S006}
\begin{aligned}
a_{31}&=t^2,&a_{32}&=t^{-2},&a_{41}&=t^{-2},&a_{42}&=t^2,&a_{43}&=1,\\
u_{31}&=0,  &u_{32}&=0,     &u_{41}&=0,     &u_{42}&=0,  &u_{43}&=\frac{1}{4}(t-t^{-1})(H-K).
\end{aligned}
\end{equation}
Then there exists a skew polynomial $\Bbb F$-algebra $\Bbb F[H,K][E;\beta_3]$ by Remark~\ref{REM}(2) and, applying Theorem~\ref{MAIN},
there exists an iterated skew polynomial $\Bbb F$-algebra
$$A=\Bbb F[H,K][E;\beta_3][F;\beta_4,\nu_4],$$
where
$$\begin{aligned}
\beta_3(H)&=t^2H,&\beta_3(K)&=t^{-2}K,&&&&\\
\beta_4(H)&=t^{-2}H,&\beta_4(K)&=t^2K,&\beta_4(E)&=E, &&\\
\nu_4(H)&=0,&\nu_4(K)&=0,&\nu_4(E)&=\frac{1}{4}(t-t^{-1})(H-K).
\end{aligned}$$
Moreover the element $HK-1$ is a central element of $A$ and $t-1$ is a regular element of $A$ and $B=A/(HK-1)A$.
Note that the semiclassical limit $A/(t-1)A$ is Poisson isomorphic to $D$ by Corollary~\ref{AD}
since all $a_{ji}$, $u_{ji}$ satisfy (\ref{MDER}).
Observe that the semiclassical limit $B/(t-1)B$ is Poisson isomorphic to $C$.

Let $0,\pm1\neq q\in\Bbb C$.
Then $t-q$ is a nonzero and nonunit in $A$ and $B$.
The deformation $B_q=B/(t-q)B$ is a nontrivial $\Bbb C$-algebra with the multiplication induced by that of $B$,
which is isomorphic to $U_q(\frak{sl}_2(\Bbb C))$ in \cite[I.3.1]{BrGo} as shown in \cite[4.5]{Jor7}.
\end{example}

\begin{prop}\label{exact}
Fix $h\in B_3=\Bbb C[x_1,x_2,x_3]$ with degree $\leq3$.
By \cite[Proposition 1.17]{JoOh}, $B_3$ becomes a Poisson algebra with Poisson bracket
\begin{equation}\label{MO1}
\{f,g\}=\text{det}\begin{pmatrix}\frac{\partial h}{\partial x_1} &\frac{\partial h}{\partial x_2}&\frac{\partial h}{\partial x_3}\\
\frac{\partial f}{\partial x_1}&\frac{\partial f}{\partial x_2}&\frac{\partial f}{\partial x_3}\\ \frac{\partial g}{\partial x_1}&\frac{\partial g}{\partial x_2}&\frac{\partial g}{\partial x_3}\end{pmatrix}
\end{equation}
for $f,g\in B_3$.
Suppose that the Poisson bracket of $B_3$ satisfies the condition (\ref{MOC1}).
Then $h$ is of the form
$$h=\lambda x_1x_2x_3+\mu x_3+f_1x_2+f_2,$$
where $\lambda,\mu\in\Bbb C$ and $f_1,f_2\in\Bbb C[x_1]$ such that $\deg f_1\leq2$ and $\deg f_2\leq3$.
\end{prop}

\begin{proof}
Note that the Poisson bracket of $B_3$ is as follows:
\begin{eqnarray}
\{x_1,x_2\}&=&\frac{\partial h}{\partial x_3}\label{PPP1}\\
\{x_1,x_3\}&=&-\frac{\partial h}{\partial x_2}\label{PPP2}\\
\{x_2,x_3\}&=&\frac{\partial h}{\partial x_1}.\label{PPP3}
\end{eqnarray}
By (\ref{PPP1}) and (\ref{MOC1}), we have that $\frac{\partial h}{\partial x_3}=\{x_1,x_2\}=-c_{21}x_1x_2-p_{21}$ and thus
\begin{equation}\label{PPP4}
h=-(c_{21}x_1x_2+p_{21})x_3+f,
\end{equation}
where $c_{21}\in\Bbb C$, $p_{21}\in\Bbb C[x_1]$ with degree $\leq2$ and $f\in\Bbb C[x_1,x_2]$ with degree $\leq3$.
By (\ref{PPP2}), (\ref{PPP4}) and (\ref{MOC1}), we have
$$c_{31}x_1x_3+p_{31}=\{x_3,x_1\}=-c_{21}x_1x_3+\frac{\partial f}{\partial x_2}$$
and thus $\frac{\partial f}{\partial x_2}=p_{31}\in\Bbb C[x_1]$.
It follows that $f=p_{31}x_2+f_2$ and thus
\begin{equation}\label{PPP5}
h=-(c_{21}x_1x_2+p_{21})x_3+p_{31}x_2+f_2
\end{equation}
by (\ref{PPP4}), where $f_2\in\Bbb C[x_1]$ such that $\deg p_{31}\leq 2$ and $\deg f_2\leq 3$.
By (\ref{PPP3}) and (\ref{PPP5}), we have that
$$-c_{32}x_2x_3-p_{32}=\{x_2,x_3\}=-c_{21}x_2x_3-p'_{21}x_3+p_{31}'x_2+f_2',$$
where $p_{21}'=\frac{\partial p_{21}}{\partial x_1},p_{31}'=\frac{\partial p_{31}}{\partial x_1},f_2'=\frac{\partial f_2}{\partial x_1}$,
and thus $p_{21}\in\Bbb C$.
Hence $h$ is of the form
$$h=\lambda x_1x_2x_3+\mu x_3+f_1x_2+f_2$$
for some $\lambda,\mu\in\Bbb C$ and $f_1,f_2\in\Bbb C[x_1]$ with $\text{deg}f_1\leq2$ and $\text{deg}f_2\leq3$, as claimed.
\end{proof}

\begin{example}\label{MMO}
Retain the notations of Proposition~\ref{exact}.
Suppose that $\text{deg} f_1=0$, namely $f_1\in\Bbb C$.
By (\ref{MO1}), $B_3$ is a Poisson algebra with the Poisson bracket
$$\{x_2,x_1\}=-\lambda x_1x_2-\mu,\,\{x_3,x_1\}=\lambda x_1x_3+f_1,\,\{x_3,x_2\}=-\lambda x_2x_3-\frac{\partial f_2}{\partial x_1}.$$
Hence $B_3$ is an iterated Poisson polynomial $\Bbb C$-algebra
$$B_3=\Bbb C[x_1][x_2;\alpha_2,\delta_2]_p[x_3;\alpha_3,\delta_3]_p$$
by \cite[1.1]{Oh8}, where
$$\begin{aligned}
&\alpha_2(x_1)=-\lambda x_1, &&\alpha_3(x_1)=\lambda x_1,  &&\alpha_3(x_2)=-\lambda x_2,\\
&\delta_2(x_1)=-\mu,         &&\delta_3(x_1)=f_1,          &&\delta_3(x_2)=-\frac{\partial f_2}{\partial x_1}.
\end{aligned}$$

Let $\Bbb F=\Bbb C[[t-1]]$ and let $U(\Bbb F)$ be the unit group of $\Bbb F$.
Note that $t-1$ is a nonzero, nonunit and non-zero-divisor of $\Bbb F$.
Fix $\widetilde{\lambda}\in U(\Bbb F)$, $\widetilde{\mu},\widetilde{f_1}\in\Bbb F$, $\widetilde{g}\in\Bbb F[x_1]$ such that
\begin{equation}\label{SS00}
\begin{aligned}
&\widetilde{\lambda}-1\in(t-1)\Bbb F,&&\widetilde{\mu},\ \widetilde{f_1}\in(t-1)\Bbb F,&&\widetilde{g}\in(t-1)\Bbb F[x_1],\\
&\frac{d\widetilde{\lambda}}{dt}|_{t=1}=\lambda,&&\frac{d\widetilde{\mu}}{dt}|_{t=1}=\mu,\ \frac{d\widetilde{f_1}}{dt}|_{t=1}
=f_1,&&\frac{d\widetilde{g}}{dt}|_{t=1}=\frac{\partial f_2}{\partial x_1}.
\end{aligned}
\end{equation}
(Such ones exist. For example, $\widetilde{\lambda}=e^{\lambda(t-1)}$, $\widetilde{\mu}=(t-1)\mu$, $\widetilde{f_1}=(t-1)f_1$,  $\widetilde{g}=(t-1)\frac{\partial f_2}{\partial x_1}$.)
Set
\begin{equation}\label{S01}
a_{21}=\widetilde{\lambda}^{-1},\,u_{21}=-\widetilde{\mu}.
\end{equation}
The $\Bbb F$-linear maps $\beta_2$ and $\nu_2$ on $\Bbb F[x_1]$ defined by
$$\beta_2(x_1)=a_{21}x_1=\widetilde{\lambda}^{-1}x_1,\,\,\nu_2(x_1)=u_{21}=-\widetilde{\mu}$$
satisfy (\ref{ENDO1}) and (\ref{ENDO2}) trivially.
Hence there exists a skew polynomial $\Bbb F$-algebra $A_2=\Bbb F[x_1][x_2;\beta_2,\nu_2]$ by Theorem~\ref{MAIN}.

Set
\begin{equation}\label{S02}
\begin{array}{ll}
a_{31}=\widetilde{\lambda},&a_{32}=\widetilde{\lambda}^{-1},\\
u_{31}=\widetilde{f_1},    &u_{32}=-\widetilde{g}.
\end{array}
\end{equation}
Since $u_{21},u_{31}\in\Bbb F$, $u_{32}\in\Bbb F[x_1]$ and $a_{31}^{-1}=a_{21}=a_{32}$,
the $\Bbb F$-linear maps $\beta_3$ and $\nu_3$ on $A_2$ subject to
$$\begin{array}{ll}
\beta_3(x_1)=a_{31}x_1=\widetilde{\lambda}x_1,&\beta_3(x_2)=a_{32}x_2=\widetilde{\lambda}^{-1}x_2,\\
\nu_3(x_1)=u_{31}=\widetilde{f_1},            &\nu_3(x_2)=u_{32}=-\widetilde{g}
\end{array}$$
satisfy (\ref{ENDO1}) and (\ref{ENDO2}).
Hence, by Theorem~\ref{MAIN}, there exists a skew polynomial $\Bbb F$-algebra
$$A_3=A_2[x_3;\beta_3,\nu_3]=\Bbb F[x_1][x_2;\beta_2,\nu_2][x_3;\beta_3,\nu_3].$$
Note that $t-1$ is a regular element in $A_3$.
Thus the semiclassical limit $A_3/(t-1)A_3$ is Poisson isomorphic to $B_3$ by Corollary~\ref{AD}
since all $a_{ji}$, $u_{ji}$ satisfy (\ref{MDER}) by (\ref{SS00}).

For every $1\neq q\in \Bbb C$, $t-q$ is a unit in $\Bbb F=\Bbb C[[t-1]]$ and thus $A_3/(t-q)A_3$ is trivial.
Hence, in order to find nontrivial deformations, we need a suitable subalgebra $A_3'$ of $A_3$ such that deformations $A_3'/(t-q)A_3'$ are nontrivial, as one sees below.

As a special case, let $\Bbb F=\Bbb C[t,t^{-1}]$ and $\lambda=-2,\,\,\mu=2,\,\,f_1=2,\,\,f_2=2x_1$.
Then
$$h=-2x_1x_2x_3+2x_3+2x_2+2x_1$$
and $B_3$ is a Poisson $\Bbb C$-algebra with the Poisson bracket
$$\{x_2,x_1\}=2x_1x_2-2,\,\,\{x_3,x_1\}=-2x_1x_3+2,\,\,\{x_3,x_2\}=2x_2x_3-2.$$
Setting
$$\widetilde{\lambda}=t^{-2},\,\,\widetilde{\mu}=t^{2}-1,\,\,\widetilde{f_1}=-(t^{-2}-1),\,\,\widetilde{g}=t^{2}-1,$$
there is an $\Bbb F$-algebra $A_3=\Bbb F[x_1][x_2;\beta_2,\nu_2][x_3;\beta_3,\nu_3]$ such that
$$\begin{aligned}
&\beta_2(x_1)=a_{21}x_1=t^2x_1,&&\beta_3(x_1)=a_{31}x_1=t^{-2}x_1,&&\beta_3(x_2)=a_{32}x_2=t^2x_2,\\
&\nu_2(x_1)=u_{21}=-(t^{2}-1), &&\nu_3(x_1)=u_{31}=-(t^{-2}-1),   &&\nu_3(x_2)=u_{32}=-(t^{2}-1).
\end{aligned}$$
Note that $A_3$ is the $\Bbb F$-algebra generated by $x_1, x_2, x_3$ subject to the relations
\begin{equation}\label{S3}
t^2x_1x_2-x_2x_1=t^2-1,\,\,t^2x_3x_1-x_1x_3=t^2-1,\,\,t^2x_2x_3-x_3x_2=t^2-1.
\end{equation}
Let $0,1\neq q\in\Bbb C$ and let $A_3^q$ be the deformation $A_3^q=A_3/(t-q)A_3$ of $B_3$.
Then $A_3^q$ is the $\Bbb C$-algebra generated by $x_1, x_2, x_3$ subject to the relations obtained from (\ref{S3}) by replacing $t$ by $q$.
Observe that the set $\{x_3^i|i=0,1,\ldots\}$ is an Ore set of $A_3^q$ by the second and the third equations of (\ref{S3}).
The localization $A_3^q[x_3^{-1}]$ of $A^q_3$ at $\{x_3^i|i=0,1,\ldots\}$ is isomorphic to $U_q(\frak{sl}_2)$ by Ito, Terwilliger and Weng \cite{ItTeWe}, which is $Y_q$ in \cite[4.5]{Jor7}.
\end{example}

\begin{example}\label{PWA}
As in \cite[2.2]{Good4}, we find a quantization and deformations of a well-known Poisson algebra
$B_k=\Bbb C[x_1,x_2,\ldots, x_{2k-1},x_{2k}]$ with Poisson bracket
$$\{f,g\}=\sum_{i=1}^k\left(-\frac{\partial f}{\partial x_{2i-1}}\frac{\partial g}{\partial x_{2i}}+\frac{\partial g}{\partial x_{2i-1}}\frac{\partial f}{\partial x_{2i}}\right),$$
which is called Poisson Weyl algebra in \cite[1.1.A]{ChPr} and \cite[1.3]{MeHaO}.
Since $B_k$ is a Poisson algebra with Poisson bracket
$$\{x_j,x_i\}=\left\{\begin{aligned} &1,&&\text{if }j=2\ell,i=2\ell-1,\\ &0,&&\text{otherwise}\end{aligned}\right.$$
for $j>i$, $B_k$ is an iterated Poisson polynomial algebra
$$B_k=\Bbb C[x_1][x_2;\delta_2]_p\ldots[x_{2k-1}]_p[x_{2k};\delta_{2k}]_p,$$
where
$$\delta_{2\ell}(x_i)=\left\{\begin{aligned} &1,& &\text{if }i=2\ell-1,\\ &0,&&\text{if }i\neq 2\ell-1.\end{aligned}\right.$$

Set $\Bbb F=\Bbb C[t]$ and let
\begin{equation}\label{WEYL2}
a_{ji}=1,\,\,u_{ji}=\left\{\begin{aligned}&t-1, &&\text{if }j=2\ell,i=2\ell-1,\\ &0,&&\text{otherwise}\end{aligned}\right.
\end{equation}
for all $1\leq i<j\leq2k$.
By Theorem~\ref{MAIN}, there exists an iterated skew polynomial $\Bbb F$-algebra
$$A_k=\Bbb F[x_1][x_2;\nu_2]\ldots[x_{2k-1}][x_{2k};\nu_{2k}],$$
where
$$\nu_{2\ell}(x_i)=\left\{\begin{aligned}&t-1,&&\text{if }i=2\ell-1,\\ &0,&&\text{if }i\neq 2\ell-1.\end{aligned}\right.$$
Thus $A_k$ is an $\Bbb F$-algebra generated by $x_1,x_2,\ldots, x_{2k-1},x_{2k}$ subject to the relations
\begin{equation}\label{YR1}
x_jx_i-x_ix_j=\left\{\begin{aligned}&t-1&&\text{if }j=2\ell,i=2\ell-1\\&0&&\text{otherwise},\end{aligned}\right.
\end{equation}
which is the algebra appearing in \cite[Proposition 3.2]{MyOh}.
For each $0\neq \lambda\in \Bbb C$, a deformation $A_\lambda=A_k/(t-1-\lambda)A_k$ is a $\Bbb C$-algebra generated by
$x_1,x_2,\ldots, x_{2k-1},x_{2k}$ subject to the relations
$$x_jx_i-x_ix_j=\left\{\begin{aligned}&\lambda,&&\text{if }j=2\ell,i=2\ell-1,\\
                                      &0,      &&\text{otherwise.}\end{aligned}\right.$$
Hence we get a family of infinite nontrivial deformations $\{A_\lambda|0\neq \lambda\in\Bbb C\}$, all of which are isomorphic to the $k$-th Weyl algebra by \cite[Proposition 3.4]{MyOh}.

Note that $t-1$ is a regular element of $A_k$.
By Corollary~\ref{AD}, the semiclassical limit $A_k/(t-1)A_k$ is Poisson isomorphic to $B_k$
since
$$a_{ji}-1\in(t-1)\Bbb F,\,\,\frac{da_{ji}}{dt}|_{t=1}=0,\,\,u_{ji}\in(t-1)A_i,\,\,\frac{du_{ji}}{dt}|_{t=1}=[\delta_j(x_i)].$$

\end{example}

\begin{example}\label{PWA2}
Let $B_k$ be the Poisson Weyl algebra given in Example~\ref{PWA}.
Set $\Bbb F=\Bbb C[[t-1]]$
and
\begin{equation}\label{WEYL}
a_{ji}=\left\{\begin{aligned} &\cos(t-1),&&\text{if $i+j$ is odd},\\  &\sec(t-1),&&\text{if $i+j$ is even,}\end{aligned}\right.\ \ \
u_{ji}=\left\{\begin{aligned}&\sin(t-1), &&\text{if }j=2\ell,i=2\ell-1,\\ &0,&&\text{otherwise}\end{aligned}\right.
\end{equation}
for all $1\leq i<j\leq 2k$.
Note that $a_{ji}, u_{ji}\in \Bbb F$ by elementary calculus.

We will show by induction on $k$ that there exists an iterated skew polynomial $\Bbb F$-algebra
$$A_k=\Bbb F[x_1][x_2;\beta_2,\nu_2]\ldots[x_{2k-1};\beta_{2k-1}][x_{2k};\beta_{2k},\nu_{2k}],$$
where
$$\beta_j(x_i)=a_{ji}x_i,\,\,\nu_j(x_i)=u_{ji}$$
for all $1\leq i<j\leq 2k$.
If $k=1$ then there exists the skew polynomial $\Bbb F$-algebra $A_1=\Bbb F[x_1][x_2;\beta_2,\nu_2]$ trivially by Theorem~\ref{MAIN}.
Suppose that $k>1$ and assume that there exists an iterated skew polynomial $\Bbb F$-algebra $A_{k-1}$.
Note that, for any positive integers $i,j,\ell$,
\begin{equation}\label{EOEO}
\begin{aligned}
&\text{$i+j$ is odd if and only if}\\
&\qquad\qquad\quad\text{($\ell+j$ is odd and $\ell+i$ is even) or ($\ell+j$ is even and $\ell+i$ is odd)}.
\end{aligned}\end{equation}
Observe that $\Bbb F$-linear maps $\beta_{2k-1}$ and $\nu_{2k-1}$ satisfy (\ref{ENDO2}) trivially
since $\nu_{2k-1}(u_{ji})=0$ and $u_{2k-1,i}=0$ for all $1\leq i<2k-1$
and that they also satisfy (\ref{ENDO1}) by (\ref{EOEO})
since $\beta_{2k-1}(u_{ji})=u_{ji}$.
Hence there exists a skew polynomial $\Bbb F$-algebra $A_{k-1}[x_{2k-1};\beta_{2k-1}]$ by Theorem~\ref{MAIN}.
For $\Bbb F$-linear maps $\beta_{2k}$ and $\nu_{2k}$, they satisfy (\ref{ENDO1}) and (\ref{ENDO2}) by (\ref{EOEO})
since $\beta_{2k}(u_{ji})=u_{ji}$ and $\nu_{2k}(u_{ji})=0$
and thus there exists $A_k=A_{k-1}[x_{2k-1};\beta_{2k-1}][x_{2k};\beta_{2k},\nu_{2k}]$ by Theorem~\ref{MAIN}.

Note that $t-1$ is a regular element of $A_k$.
By Corollary~\ref{AD}, the semiclassical limit $A_k/(t-1)A_k$ is Poisson isomorphic to  $B_k$
since
$$a_{ji}-1\in(t-1)\Bbb F,\,\,\frac{da_{ji}}{dt}|_{t=1}=0,\,\,u_{ji}\in(t-1)A_i,\,\,\frac{du_{ji}}{dt}|_{t=1}=[\delta_j(x_i)]$$
by elementary calculus.

Note that $A_k$ is an $\Bbb F$-algebra generated by $x_1,x_2,\ldots, x_{2k-1}, x_{2k}$ subject to the relations
\begin{equation}\label{MODD}
\begin{aligned}
x_{2\ell}x_{2\ell -1}-\cos(t-1)x_{2\ell -1}x_{2\ell}&=\sin(t-1), &&(\ell=1,\ldots,k),\\
x_jx_i-\sec(t-1)x_ix_j&=0, &&(\text{$i<j$, \ $i+j$ is even}),\\
x_jx_i-\cos(t-1)x_ix_j&=0, &&\left({\begin{aligned}&\text{$i<j$,\  $i+j$ is odd, }\\
&\text{if $j=2\ell$ then $i\neq 2\ell-1$}\end{aligned}}\right).
\end{aligned}\end{equation}
For any $0\neq q\in\Bbb C$, $t-1-q$ is a unit in $A_k$ and thus $A_k/(t-1-q)A_k$ is trivial.
It follows that we need an appropriate subalgebra of $A_k$ to find a nontrivial deformation.
For instance, let $A_k'$ be the $\Bbb C[t]$-subalgebra of $A_k$ generated by $x_1, x_2,\ldots, x_{2k-1}, x_{2k}$.
Evaluating $A_k'$ to $\pi$ at $t-1$, we have a deformation $A_k^{\pi}$ which is the $\Bbb C$-algebra generated by $x_1, x_2,\ldots, x_{2k-1}, x_{2k}$
subject to the relations
$$x_jx_i+x_ix_j=0\,\,(j>i)$$
by (\ref{MODD}).
In this case the evaluation map $\varphi$ from $A_k'$ onto $A_k^\pi$ defined by $f\mapsto f|_{t-1=\pi}$ is a $\Bbb C$-algebra epimorphism and thus $A_k'/\ker \varphi\cong A_k^\pi$.
\end{example}

\begin{example}
The commutative $\Bbb C$-algebra $B=\Bbb C[x_1,\ldots,x_n]$ is a Poisson $\Bbb C$-algebra with Poisson bracket
$$\{x_j,x_i\}=x_ix_j$$
for all $1\leq i<j\leq n$ by \cite[Example 4.5]{Good3}.
Note that $B$ is an iterated Poisson polynomial $\Bbb C$-algebra
$$B=\Bbb C[x_1][x_2;\alpha_2]_p\ldots[x_n;\alpha_n]_p,$$
where $\alpha_j(x_i)=x_i$ for all $1\leq i<j\leq n$.

Set $\Bbb F=\Bbb C[t]$ and $a_{ki}=t$ for $1\leq i<k\leq n$. Then, by Remark~\ref{REM}(2) and Theorem~\ref{MAIN}, there exists an iterated skew polynomial $\Bbb F$-algebra
$$A=\Bbb F[x_1][x_2;\beta_2]\ldots[x_n;\beta_n],$$
where $\beta_k(x_i)=a_{ki}x_i$ for all $1\leq i<k\leq n$.
Note that $t-1$ is a regular element of $A$.
By Corollary~\ref{AD}, $A/(t-1)A$ is Poisson isomorphic to $B$
since
$$a_{ki}-1\in(t-1)\Bbb F,\,\,\frac{d a_{ki}}{dt}|_{t=1}=1.$$

Let $0,1\neq q\in\Bbb C$.
The deformation $A_q=A/(t-q)A$ of $B$ is the $\Bbb C$-algebra generated by $x_1,\ldots, x_n$ subject to the relations
$$x_jx_i=qx_ix_j$$
for all $1\leq i<j\leq n$, which is the coordinate ring $\mathcal{O}_{q}(\Bbb C^n)$ of quantum affine $n$-space in \cite[I.2.1]{BrGo}.
\end{example}

\begin{example}
A Poisson $2\times 2$-matrices algebra is the coordinate ring of $2\times2$-matrices, $\mathcal{O}(M_2(\Bbb C))=\Bbb C[x, y, z, w]$, with Poisson bracket
$$\begin{aligned}
&\{x,y\}=xy, &&\{x,z\}=xz, &&\{x,w\}=2yz,\\
&\{y,z\}=0,  &&\{y,w\}=yw, &&\{z,w\}=zw
\end{aligned}$$
by \cite[Example 4.9]{Good3}.
Note that $\mathcal{O}(M_2(\Bbb C))$ is an iterated Poisson polynomial $\Bbb C$-algebra
$$\mathcal{O}(M_2(\Bbb C))=\Bbb C[y][z][x;\alpha_3]_p[w;\alpha_4,\delta_4]_p,$$
where
$$\begin{aligned}
&\alpha_3(y)=y, &&\alpha_3(z)=z, &&\\
&\alpha_4(y)=-y,&&\alpha_4(z)=-z,&&\alpha_4(x)=0,\\
&\delta_4(y)=0, &&\delta_4(z)=0, &&\delta_4(x)=-2yz.
\end{aligned}$$

Set $\Bbb F=\Bbb C[t, t^{-1}]$ and
\begin{equation}\label{EX31}
\begin{aligned}
&a_{31}=t,          &&a_{32}=t,          &&\\
&u_{31}=0,          &&u_{32}=0,          &&\\
&a_{41}=a_{31}^{-1},&&a_{42}=a_{32}^{-1},&&a_{43}=1,  \\
&u_{41}=0,          &&u_{42}=0,          &&u_{43}=-(t-t^{-1})yz.
\end{aligned}
\end{equation}
We show that there exists an iterated skew polynomial $\Bbb F$-algebra
$$A=\Bbb F[y,z][x;\beta_3][w;\beta_4,\nu_4],$$
where
\begin{equation}\label{EX32}
\begin{aligned}
&\beta_3(y)=a_{31}y,     &&\beta_3(z)=a_{32}z,     &&\\
&\beta_4(y)=a_{31}^{-1}y,&&\beta_4(z)=a_{32}^{-1}z,&&\beta_4(x)=a_{43}x,\\
&\nu_4(y)=0,             &&\nu_4(z)=0,             &&\nu_4(x)=u_{43}.
\end{aligned}
\end{equation}
By Remark~\ref{REM}(2) and Theorem~\ref{MAIN}, there exists a skew polynomial $\Bbb F$-algebra $\Bbb F[y,z][x;\beta_3]$.
Note that $\Bbb F[y,z]$ is commutative and $u_{43}\in\Bbb F[y,z]$, $a_{42}a_{32}=a_{41}a_{31}=1$.
Hence $\Bbb F$-linear maps $\beta_4$ and $\nu_4$ satisfy (\ref{ENDO1}) and (\ref{ENDO2})
and thus there exists an iterated skew polynomial $\Bbb F$-algebra $A$ by Theorem~\ref{MAIN}.
Note that $t-1$ is a regular element of $A$.
Hence the semiclassical limit $A/(t-1)A$ is Poisson isomorphic to $\mathcal{O}(M_2(\Bbb C))$ by Corollary~\ref{AD}
since all $a_{ji}$, $u_{ji}$ satisfy (\ref{MDER}).

The deformation
$$A_q=A/(t-q)A,\,\,(0,1\neq q\in\Bbb C)$$
with multiplication induced by that of $A$ is the $\Bbb C$-algebra generated by $x,y,z,w$ subject to the relations
$$\begin{aligned}
zy&=yz, &xy&=qyx,&xz&=qzx,\\
yw&=qwy,&zw&=qwz,&xw-wx&=(q-q^{-1})yz.
\end{aligned}$$
Following \cite[I.1.7]{BrGo}, $A_q$ is the quantum $2\times2$-matrices algebra $\mathcal{O}_{q}(M_2(\Bbb C))$ as expected.
\end{example}

\noindent
{\bf Acknowledgments} The authors would like to appreciate a referee for indicating valuable comments and mistakes.
The second author is supported by National Research Foundation of Korea, NRF-2017R1A2B4008388.


\bibliographystyle{amsplain}

\providecommand{\bysame}{\leavevmode\hbox to3em{\hrulefill}\thinspace}
\providecommand{\MR}{\relax\ifhmode\unskip\space\fi MR }
\providecommand{\MRhref}[2]{%
  \href{http://www.ams.org/mathscinet-getitem?mr=#1}{#2}
}
\providecommand{\href}[2]{#2}

\end{document}